\documentclass[11pt, a4paper]{amsart}
\usepackage{amssymb,latexsym,amsmath, graphicx} 
\usepackage{enumerate,enumitem,color,hyperref} 
\usepackage{mathrsfs} 
\usepackage{calligra}
\usepackage[T1]{fontenc}
\usepackage{array}

\input{xypic}
\xyoption{all}

\def\C{\mathbb{C}}
\def\Z{\mathbb{Z}}
\def\Q{\mathbb{Q}}

\def\fP{\mathfrak{P}}  

\def\SL{\mathrm{SL}}

\newtheorem{theorem}{Theorem}[section]

\newtheorem{definition}{Definition}[section]
\newtheorem{lemma}[theorem]{Lemma}
\newtheorem{proposition}[theorem]{Proposition}

\newtheorem{conjecture}{Conjecture}[section]

\usepackage{subfiles}

\title{Congruences for the ratios of Rankin--Selberg $L$-functions}
\author{\bf P. Narayanan \ \ \ \& \ \ \ A. Raghuram}
\date{\today}
\subjclass[2020]{11F33; 11F30, 11F67}

\address{Dept.\,of Mathematics, Indian Institute of Science Education and Research, Dr.\,Homi Bhabha Road, Pune 411008, INDIA.}

\email{narayanan.p@students.iiserpune.ac.in}

\address{Dept.\,of Mathematics, Fordham University at Lincoln Center, New York, NY 10023, USA.} 

\email{araghuram@fordham.edu}

\keywords{Congruences, Special Values, Rankin--Selberg $L$-functions}

\begin{document}

\begin{abstract}
A well-known principle states that a congruence between objects should give rise to a corresponding congruence between the 
special values of $L$-functions attached to these 
objects. We computationally investigate this principle for Rankin--Selberg $L$-functions attached to pairs of holomorphic cuspforms, and 
formulate a precise conjecture in general. 
\end{abstract}

\maketitle

{\small \tableofcontents }

\section{Introduction}

A well-known principle with origins in Iwasawa theory states that a congruence between objects should translate to a congruence between the special values of 
$L$-functions attached to these 
objects. The reader is referred to Vatsal \cite{vatsal} for an instance of this principle, and for a brief discussion of its historical origins. In this article we computationally 
investigate this principle for Rankin--Selberg $L$-functions attached to pairs of holomorphic cuspforms. 

\smallskip

Suppose we have two primitive 
holomorphic cusp forms $f, f' \in S_k(\Gamma_1(N))$; by a form being primitive, one means it is an eigenform, a newform, and is normalised by taking the first 
Fourier coefficient to be $1$. 
Suppose 
$f(z) = \sum_{n=1}^\infty a(n,f) e^{2\pi i n z}$ and $f'(z) = \sum_{n=1}^\infty a(n,f') e^{2\pi i n z}$ are their Fourier expansions. 
Let $\Q(f)$ (resp., $\Q(f')$) denote the number field 
generated by the Fourier coefficients of $f$ (resp., $f'$), and take $F$ to be any number field containing $\Q(f)$ and $\Q(f')$. 
Let $\fP$ be a prime ideal of the ring of integers of 
$F$. We say $f$ is congruent to $f'$ modulo $\fP$ if $a(n,f) \equiv a(n,f') \pmod{\fP}$ for all $n \geq 1$; denote this congruence as $f \equiv f' \pmod{\fP}.$ 

\smallskip

Suppose $g \in S_l(\Gamma_1(M))$. 
For the Rankin--Selberg $L$-function $L(s, f \times g)$ to have critical points one needs $k \neq l$. 
Assume for the moment that $k > l$. Then the set of critical points of $L(s, f \times g)$ is $\{m \in \Z : l \leq m \leq k-1\}$. Since $f$ and $f'$ have the same weight, 
this set is also the set of critical points for $L(s, f' \times g).$
For any of these critical points, one expects a congruence between the algebraic parts of $L(m, f \times g)$ and $L(m, f' \times g).$ 
Of course, one has to be careful in clarifying what one means by the algebraic part of an $L$-value. Keeping in mind the shape of the rationality
results due to Shimura for these $L$-values (see Thm.\,\ref{thm:main-theorem}), one observes that if one had at least two critical points, then the ratio 
$L(m, f \times g)/L(m+1, f \times g)$ of successive critical values for the completed $L$-function is algebraic and 
lies in the compositum of the number fields $\Q(f)$ and $\Q(g)$. 
Let us enlarge the number field $F$ above also to contain $\Q(g)$. The purpose of this paper is to investigate the principle: 
$$
f \equiv f' \pmod{\fP} \ \ \implies \ \ \frac{L(m, f \times g)}{L(m+1, f \times g)} \equiv \frac{L(m, f' \times g)}{L(m+1, f' \times g)} \pmod{\fP}.
$$
It is possible that the ratios of $L$-vaues are not integral in $F$; the congruence is to be interpreted as the difference of the 
ratios has a positive $\mathfrak{P}$-adic valuation.

\smallskip
An obvious variation to consider is to take $k < l,$ i.e., freeze the larger weight cuspform $f$ and take two congruent cuspforms 
$g,g'$ of smaller weight. One may also consider a 
congruence between a cuspform and an Eisenstein series such as the celebrated Ramanujan congruence.  
If $g$ is an Eisenstein series then $L(s, f \times g)$ is a product of the form
$L(s, f \otimes \chi_1)\cdot L(s, f \otimes \chi_1)$ for algebraic Hecke characters $\chi_1$ and $\chi_2$ of $\Q$. 

\smallskip
In this article, we verify that the expected congruences hold in the following cases: 

\begin{enumerate}
\item $f, f' \in S_{24}(\SL_2(\Z))$, $\fP = 144169$, $g \in S_{12}(\SL_2(\Z))$.

\smallskip
\item $f, f' \in S_{30}(\SL_2(\Z))$, $\fP = 51349$, $g \in S_{12}(\SL_2(\Z))$.

\smallskip
\item  $f, f' \in S_{13}(\Gamma_0(3), \chi)$, $\fP = 13$, $g \in S_{6}(\Gamma_0(3))$.

\smallskip
\item $f \in S_{26}(\SL_2(\Z))$, and $g, g' \in S_{13}(\Gamma_0(3), \chi)$, $\fP = 13$. 

\smallskip
\item $f \in S_{24}(\SL_2(\Z))$, $g \in S_{12}(\SL_2(\Z))$, $g'$ a weight $12$ Eisenstein series, $\fP = 691.$ \\ (This is the Ramanujan congruence.) 
\end{enumerate}
In (1) and (2), $f'$ is a Galois conjugate of $f$, but in (3) $f'$ is not a Galois conjugate of $f$. 

\smallskip

The computations rely on two algorithms. 
The first algorithm, reviewed in \ref{subsection: algorith-gl2-x-gl2}, computes the special values of Rankin-Selberg $L$-functions, and is essentially due to Shimura and Hida. 
Roughly speaking, it relies on the interpretation due to Shimura \cite{shimura} of the $L$-value as a Petersson inner product of a cusp form with a holomorphic 
projection, and the computation due to Hida \cite{hida} of the holomorphic projection recursively from the given data and an appropriate Eisenstein series. 
The second well-known algorithm, reviewed in \ref{subsection: algorith-gl2}, 
computes the $L$-values of modular forms from the Hecke-equivariant pairing between a space of cusp forms 
and a suitable space of modular symbols. It should be mentioned that there is also an algorithm by Tim Dokchitser \cite{dokchitser} to compute the $L$-values of modular 
forms numerically.
All the computations in this paper are done on SAGE \cite{sage}.

\medskip

Based on these examples, we formulate a conjecture in Conj.\,\ref{conj:main}. In the companion article \cite{narayanan-raghuram}, we prove a 
variation of this conjecture using the machinery of Eisenstein cohomology as developed by 
Harder and the second author in \cite{harder-raghuram}. As the proofs in \cite{harder-raghuram} involve a cohomological interpretation of 
Langlands's theorem on the constant term of an Eisenstein series, the results and proofs in \cite{narayanan-raghuram} are completely independent of 
the computational techniques this paper.

\bigskip

\section{Preliminaries and the main algorithms}

In this section, we begin by reviewing some preliminaries from Shimura \cite{shimura} on modular forms 
and the special values of Rankin-Selberg $L$-functions and some results of Hida \cite{hida} on the holomorphic projection. 
We present two algorithms which will be used to verify the congruences mentioned in the introduction.

\subsection{Review of some results of Shimura and Hida}

Let $\Gamma_1(N)$ be the standard Hecke congruence subgroup of $\SL_2(\mathbb{Z})$; the volume of  
 its fundamental domain $\Phi_1(N)$ with respect to the measure $y^{-2}dxdy$ will be denoted $\mu(\Phi_1(N)).$ 
 Let $M_k(\Gamma_1(N))$ be the space of holomorphic modular forms of weight $k$ for  $\Gamma_1(N)$. 
 Sometimes $M_k(\Gamma_1(N))$ is abbreviated as $M_k(N).$
For $f, \, h \in M_k(\Gamma_1(N))$ such that $fh$ is a cusp form, define the Petersson inner product $\langle f, h \rangle$ by: 
\begin{equation*}
		\langle f, h \rangle \ := \ \mu(\Phi_1(N))^{-1} \int_{\Phi_1(N)}\overline{f(z)}h(z)y^{k-2} \, dx dy .
\end{equation*}

Let $\chi$ be a Dirichlet character modulo $N$. 
Let $M_k(N,\chi)$ (resp., $S_k(N,\chi)$) denote the space of holomorphic modular (resp., cusp) forms of weight $k$ for $\Gamma_0(N)$ 
with nebentypus character $\chi$. 
Assume $l < k$.  
Suppose $f = \sum_{n=1}^\infty a(n,f) q^n \in S_k(N,\chi)$ and $g = \sum_{n=0}^\infty a(n,g) q^n \in M_l(N,\psi),$ where $q = e^{2\pi i z}$, 
then the Rankin--Selberg convolution is defined as the Dirichlet series: 
$$    
D (s, f\times g) \ := \ \sum_{n=1}^\infty a(n,f)a(n, g) n^{-s} .
$$

\smallskip

For an integer $\lambda \geq 0$ and a Dirichlet character $\omega $ modulo $N$ such that $\omega(-1) =(-1)^\lambda $, define an Eisenstein series: 
\begin{align*}
    E^*_{\lambda, N}(z,s,\omega) &= \sum_{\gamma \in \Gamma_\infty \backslash \Gamma_0(N)} 
    \omega(d)(cz +d )^{-\lambda}|cz+d|^{-2s} , \quad (\gamma = \begin{pmatrix} a & b \\ c & d\end{pmatrix}),
\end{align*}
where $\Gamma_\infty = \bigl\{\pm\begin{pmatrix}1 & m \\ 0 & 1 \end{pmatrix} | m \in \mathbb{Z}\bigr\}$. 
For the same $\lambda$ and $\omega,$ define: 
\begin{equation*}
    E_{\lambda, N}(z,s,\omega) = \sum_{(0,0)\neq(m,n) \in \mathbb{Z}^2} \omega(n)(mNz+n)^{-\lambda}|mNz + n|^{-2s}.
\end{equation*}
The following relation holds between $E_{\lambda, N}^*$ and $E_{\lambda, N}$: 
\begin{equation*}
    E_{\lambda,N}(z,s,\omega) = 2 L_N(2s+\lambda,\omega)E^*_{\lambda,N}(z,s,\omega),
\end{equation*}
where $L_N(s,\omega)$ denotes the Dirichlet $L$-series $\sum_{n=1}^\infty \omega(n)/n^s$ and the subscript $N$ means that 
$\omega(n) = 0$ if ${\rm gcd}(n,N) \neq 1.$ If $\omega$ is primitive modulo $N$ then this condition is automatic and the subscript may be dropped. 
One has the following integral representation: 
\begin{multline*}
    2(4\pi)^{-s}\Gamma(s)L_N(2s+2-k-l,\chi\psi) D(s,f\times g) \ = \\ 
    \int_{\Phi_0(N)} \bar{f_\rho}g \cdot E_{k-l,N}(z,s+1-k,\chi\psi)\, y^{s-1}\, dxdy. 
\end{multline*}
Here $\Phi_0(N)$ is the fundamental domain for $\Gamma_0(N) \backslash \mathbb{H}$ and $f_\rho(z) = \overline{f(-\overline{z})} = \sum_{n=1}^\infty \overline{a_n}e(nz)$. Define $E_{\lambda, N}^*(z,\omega) = E_{\lambda, N}^*(z,0,\omega)$.

\smallskip

The Maa\ss--Shimura differential operators on the space of smooth functions on the Poincar\'e upper half plane $\mathbb{H}$ are defined by:
$$
\delta_\lambda  = \frac{1}{2\pi i} \left( \frac{\lambda}{2iy} + \frac{\partial}{\partial z}\right), \ \ 0 < \lambda \in \mathbb{Z}, \quad {\rm and} \quad 
\delta_{\lambda}^{(r)} = \delta_{\lambda+2r-2}\cdots \delta_{\lambda +2}\delta_{\lambda}, \ \ 0\leq r \in \mathbb{Z}. 
$$
It is understood that $\delta_{\lambda}^{(0)}$ is the identity operator. Define $d = \frac{1}{2 \pi i}\frac{\partial}{\partial z} = q \frac{d}{dq}.$ Then 
\begin{align}
\delta_{\lambda}^{(r)} \ = \ \sum_{0\leq t \leq r} \binom{r}{t} \frac{\Gamma(\lambda + r)}{\Gamma(\lambda +t)} (-4\pi y)^{t-r} d^t, 
\label{eqn:delta_lambda-identity}
\end{align}
where the symbol $d^{0}$ means the identity operator; see \cite{shimura}. The action of these operators on $1/4\pi y$ is given by the following

\begin{lemma}
$$
d^n\left(\frac{1}{4 \pi y}\right) = \frac{n!}{(4 \pi y)^{n+1}}, \quad n=0,1,2,\dots, 
$$
$$
\delta_{\lambda}^{(r)} \left( \frac{1}{4 \pi y} \right) = 
(-4\pi y)^{-(r+1)}\cdot \sum_{t=0}^{r} \binom{r}{t}(-1)^{t+1} \frac{\Gamma(\lambda+r)\Gamma(t+1)}{\Gamma(\lambda + t)}. 
$$
\label{lemma: action-of-d-on-1-over-4piy}
\end{lemma}

\begin{proof}
The first is proved by induction on $n$; the second from the first and \eqref{eqn:delta_lambda-identity}. 
\end{proof}

Let $A_r$ denote the set of all functions of the form $h(z)= \sum_{\nu = 0}^r y^{-\nu} g_\nu(z)$ with holomorphic functions $g_\nu$ on $\mathbb{H}$ which have 
Fourier expansions $g_\nu(z) = \sum_{n=0}^\infty b_{\nu n}q^n$. Elements of $A_r$ may be called {\it nearly holomorphic modular forms.} 
The following lemma is \cite[Lem.\,7]{shimura}.

\begin{lemma}\label{lemma: hol-projection}
		For a positive integer $k>2r$ and a Dirichlet $\chi$ modulo $N$, suppose that an element $h$ of $A_r$ satisfies
\begin{equation}
		\begin{split}
		h(\gamma(z))(cz+d)^{-k} \in A_r\:\: \mbox{for all}\:\: \gamma =
		\begin{pmatrix}
				a & b \\ c & d
		\end{pmatrix} \in SL_2(\mathbb{Z})\\
		h(\gamma(z))(cz+d)^{-k} = \chi(d) h(z)\:\: \mbox{for all}\:\: \gamma = 
		\begin{pmatrix}
				a & b \\ c & d	
		\end{pmatrix} \in \Gamma_0(N). 
\end{split}
\end{equation}
Then $h(z) = \sum_{\nu = 0}^r \delta_{k-2\nu}^{(\nu)}h_\nu,$ with $h_\nu$ of $M_{k-2\nu}(N,\chi)$ being uniquley determined by $h$.
\end{lemma}

The following orthogonality relation is \cite[Lem.\,6]{shimura}. 

\begin{lemma}\label{lemma: zero-lemma}
	Suppose $f\in S_k(N,\chi)$, $g \in M_l(N,\bar{\chi}),$ and $k=l+2r$ with a positive integer $r$. Then $\langle f_\rho, \delta_l^{(r)}g \rangle=0 $.
\end{lemma}

As a consequence of the above two lemmas one has:

\begin{lemma}\label{lemma: fgdelta_fh0}
$\langle f_\rho, g \cdot \delta_{\lambda}^{(r)}E_{\lambda, N}\rangle \ = \ \langle f_\rho, h_0 \rangle,$ where 
$h_0$ is the holomorphic projection of the nearly holomorphic modular form $g \cdot \delta_\lambda^{(r)} E_{\lambda, N}.$
\end{lemma}

Combining all of the observations from before we get

\begin{theorem} \label{thm: l-value-and-petersson-inner-product}
Suppose $f \in S_k(N,\chi),$ $g \in M_l(N,\psi)$, and $l+2r < k$ with a non-negative integer $r$. Then
$$
	D(k-1-r,f\times g) \ = \ c_r \pi^k \langle f_\rho , \ g.\delta_{\lambda}^{(r)}E_{\lambda,N}^*(z,\chi\psi) \rangle = \langle f_\rho , h_0 \rangle, 
$$
where $\lambda = k - l -2r$, and $c_r$ is the rational number explicitly given by: 
$$
    c_r \ = \ \frac{\Gamma(k-l-2r)}{\Gamma(k-1-r)\Gamma(k-l-r)}.\frac{(-1)^r4^{k-1}N}{3} \prod_{p | N} (1 + p^{-1}), 
$$
where the product is taken over all prime factors $p$ of $N$.
\end{theorem}

\smallskip

When $\lambda =2$ and $\omega = 1\!\!1$ is the trivial character modulo $N$, then $E_{\lambda,N}^*(z,\omega)$ is not holomorphic at $\infty$; 
its Fourier expansion at $\infty$ is given by 
$$
E_{2, N}^*(z, 1\!\!1) \ = \ \frac{c}{4\pi y} + \sum_{n=0}^\infty c_n q^n, 
$$ 
where $c$ and all the $c_n$ are in $\Q$. This is due to Hecke; see, for example, Miyake \cite[p.\,288]{miyake}. 

\smallskip

Let $f \in S_k(N,\chi)$ and $g \in M_k(N,\psi)$. Then as noted above
$$
E_{\lambda, N}^*(z, \chi\psi) \ = \ \frac{c}{4\pi y}+ \sum_{n=0}^\infty c_n q^n \ =: \ \frac{c}{4\pi y} + E, 
$$ 
with $c \neq 0$ only when $\lambda = 2$ and $\chi\psi = 1\!\!1$. Then, by Lemma \ref{lemma: action-of-d-on-1-over-4piy} and Lemma \ref{lemma: hol-projection} 
\begin{multline}
g.\delta_\lambda^{(r)}E_{\lambda, N}^*= g\cdot \delta_\lambda^{(r)}\left(\frac{c}{4\pi y} +  E\right) \\ 
= g\cdot\frac{c\cdot c'}{(-4 \pi y)^{r+1}} + g\cdot \delta_\lambda^{(r)}(E) = 
\begin{cases} \sum_{\nu = 0}^{r+1} \delta_{k-2\nu}^{(\nu)}h_\nu & \mbox{if $\chi \psi = 1\!\!1, \lambda = 2$,} \\
\sum_{\nu = 0}^r \delta_{k-2\nu}^{(\nu)}h_\nu  & \text{otherwise},
\end{cases} \label{eqn:relations-hi}
\end{multline}
where $h_\nu \in M_{k-2\nu}(N, \chi\cdot \psi)$ and $c' \in \mathbb{Q}$ is obtained from Lemma \ref{lemma: action-of-d-on-1-over-4piy}.

\begin{lemma}\label{lemma:h-recursive-relations}
The functions $h_\nu$, as in \eqref{eqn:relations-hi}, satisfy the following recursive relations: 
\begin{enumerate}
\item[(i)] When $\lambda = 2$ and $\chi\psi = 1\!\!1$ then 
\begin{equation*}
\frac{\Gamma(k- (r+1))}{\Gamma(k - 2(r+1))}h_{r+1} = cg\sum_{t=0}^{r} \binom{r}{t}(-1)^{t+1} \frac{\Gamma(\lambda+r)\Gamma(t+1)}{\Gamma(\lambda + t)}
\end{equation*}
and
\begin{multline*}
\frac{\Gamma(k-\nu)}{\Gamma(k-2\nu)} h_{\nu} = 
\binom{r}{r-\nu}\frac{\Gamma(\lambda+r)}{\Gamma(\lambda+r-\nu)}g d^{r-\nu} E_{\lambda, N}^*(z,\chi\psi) \\ 
- \sum_{n=\nu +1}^{r+1}\binom{n}{n-\nu} \frac{\Gamma(k-n)}{\Gamma(k-n-\nu)}d^{n-j} h_n \qquad \text{for $\nu=0,1,\dots, r$}. 
\end{multline*} 

\item[(ii)] When $\lambda \neq 2$ or $\chi \psi \neq 1\!\!1$ then
\begin{multline*}
\frac{\Gamma(k-\nu)}{\Gamma(k-2\nu)} h_{\nu} = 
\binom{r}{r-\nu}\frac{\Gamma(\lambda+r)}{\Gamma(\lambda+r-\nu)}g d^{r-\nu} E_{\lambda, N}^*(z,\chi\psi) \\ 
- \sum_{n=\nu+1}^{r}\binom{n}{n-\nu} \frac{\Gamma(k-n)}{\Gamma(k-n-\nu)}d^{n-j} h_n \qquad \text{for $\nu=0,1,\dots, r$}.
\end{multline*}
\end{enumerate}
\end{lemma}

\begin{proof}
Statement $(ii)$ is proved in Hida \cite[p.177]{hida}. For $(i)$, when $\lambda = 2$ and $\chi\psi = \mbox{trivial}$, 
\begin{equation*}
  g.\delta_\lambda^{(r)}E_{\lambda, N}^*= g\cdot \delta_\lambda^{(r)}\left(\frac{c}{4\pi y} +  E\right)= g\cdot \delta_\lambda^{(r)}\left(\frac{c}{4\pi y}\right) + g\cdot\delta_\lambda^{(r)}\left( E\right)= \sum_{n = 0}^{r+1} \delta_{k - 2n}^{(n)} h_{n}.
\end{equation*}
Use \eqref{eqn:delta_lambda-identity} and expand both sides of the above equation to get:
\begin{equation*}
g \cdot \sum_{t = 0}^{r} \binom{r}{t} \frac{\Gamma(\lambda + r)}{\Gamma(\lambda+t)} (-4\pi y )^{t - r} d^t\left(\frac{c}{4 \pi y} + E\right) 
    \\= \sum_{n = 0}^{r+1}\sum_{t = 0}^{n} \binom{n}{t} \frac{\Gamma(k-n)}{\Gamma(k-2n+t)}(-4 \pi y)^{t-n} (d^t h_n).
\end{equation*}
The left hand side of the equation becomes
\begin{equation*}
c g\cdot(-4\pi y)^{-(r+1)}\cdot \sum_{t=0}^{r} \binom{r}{t}(-1)^{t+1} \frac{\Gamma(\lambda+r)\Gamma(t+1)}{\Gamma(\lambda + t)} + g\cdot \delta_{\lambda}^{(r)}E.
\end{equation*}
By comparing the coefficients of $(-4\pi y)^{t}$ for $t=0,1,\dots, r+1,$ one gets the desired equality.
\end{proof}

\bigskip
For a nomalized eigenform $f \in S_k(N,\chi)$, and a modular form $g \in M_l(N,\psi)$, define the completed
Rankin--Selberg $L$-function (with a `classical' normalization) as: 
\begin{equation}\label{eqn:abelian-infinity}
	L(s,f\times g) = L_\infty(s, f\times g) L_N(2s+2-k-l,\chi \psi) D(s, f \times g),
\end{equation} 
where $L_\infty(s, f\times g) = (2\pi)^{-2s}\Gamma(s)\Gamma(s+1-l)$ is the archimdean factor and for a 
Dirichlet character $\omega$ modulo $N$ put $L_N(s, \omega) = \sum_{n=1}^\infty \omega(n)n^{-s}$ with $\omega(n) =0 $ if $(n,N) \neq 1.$
If $g$ is also a normalized eigenform in $S_l(N, \psi),$ then $L(s,f\times g)$ satisfies a functional equation
\begin{equation}
	L(s, f\times g) \approx L(k+l -1-s,f^\rho, g^\rho).
\end{equation} 
See Shimura \cite[Section 3]{shimura} and Hida \cite[Section 9]{hida}. The line of symmetry for this functional equation is $\text{Re}(s) = (k+l-1)/2.$  
An integer $m$ is {\it critical} for $L(s,f\times g),$ if the Gamma factors on both sides of the functional equation are finite at $s = m,$ i.e., 
if $\Gamma(m)\Gamma(m+1-l)$ and $\Gamma(k+l -1-m)\Gamma(k -m)$ are finite. 
Therefore the critical set is:
$$
\{m \in \mathbb{Z} \,|\, l \leq m \leq k-1\}.
$$ 
The following result of Shimura gives an algebraicity theorem for the values 
of the Rankin--Selberg convolution at critical points to the right of the line of symmetry:

\smallskip
\begin{theorem}
\label{thm:main-theorem}
		Let $f$ be a normalized eigenform of $S_k(N)$, $g$ an element of $G_l(N)$, and $m$ a positive integer.  Suppose $l<k$ and $\lfloor \frac{1}{2}(k+l-1) \rfloor < m < k$. Then $\pi^{-k}\langle f,f\rangle^{-1}D(m,f\times g)$ 
		belongs to $\Q(f)\Q(g)$, the compositum of $\Q(f)$ and $\Q(g)$. Moreover, for every automorphism $\sigma$ of $\mathbb{C}$, we have
		\begin{equation*}
				\sigma\left(\pi^{-k}\langle f,f \rangle^{-1} D(m,f\times g) \right) \ = \ 
				\pi^{-k}\langle f^\sigma, f^\sigma \rangle^{-1}D(m,f^\sigma\times g^\sigma). 
		\end{equation*}
\end{theorem} 
One deduces that for an integer $m$ with $\frac{1}{2}(k+l-2) < m < m+1 < k$ then 
$$
	\frac{D(m, f\times g)}{D(m+1, f\times g)}  \in \mathbb{Q}(f)\mathbb{Q}(g).
$$

\bigskip

The parameters determining the abelian part and the infinite part of $L(s,f\times g)$ in \eqref{eqn:abelian-infinity} 
depend only on the weights and the nebentype characters involved. For $f,f' \in S_k(N,\chi)$ and $g \in M_l(N,\psi)$ one has 
\begin{equation}
L_{\infty}(s,f\times g) L_N(2s+2-k-l, \chi \psi) = L_{\infty}(s,f'\times g) L_N(2s+2-k-l, \chi \psi).
\label{eqn: infinite-and-abelian-L-function}
\end{equation}
For brevity, let $\omega = \chi\psi.$ 
Suppose $\omega$ is a Dirichlet character modulo $N$ which is primitive. 
Then $L_N(s, \omega) = L(s,\omega) $. Let  $\epsilon \in \{0,1\}$ such that $\omega(-1) = (-1)^\epsilon$. 
For $m\in \mathbb{N}$ and $m \equiv \epsilon \pmod{2}$ it is well known that
\begin{equation*}
	L(m, \omega) = (-1)^{1+ (m-\epsilon)/2} \frac{\mathfrak{g}(\omega)}{2\sqrt{-1}^\epsilon} \Bigl( \frac{2\pi}{N}\Bigr)^m \frac{B_{m,\bar{\omega}}}{m!},
\end{equation*} 
where $\mathfrak{g}(\omega)$ is the Gauss sum associated to $\omega$ and $B_{m,\bar{\omega}}$ is the generalized 
Bernoulli number for the character $\bar{\omega} = \omega^{-1}$. See Neukirch \cite{neukirch}. From this it follows that
\begin{equation}
\label{eqn:ratio-of-dirichlet}
	\frac{L(m,\omega)}{L(m+2, \omega)} \ = \ -(m+2)(m+1)\Bigl( \frac{N}{2\pi}\Bigr)^2 \frac{B_{m,\bar{\omega}}}{B_{m+2, \bar{\omega}}}. 
\end{equation}
If $\omega$ modulo $N$ is not primitive take $\omega^{\text{prim}}$ to be the primitive Dirichlet character defined by $\omega$. 
Then $L_N(s, \omega) = \prod_{p| N}\left(1 - \omega^{\text{prim}}(p)p^{-s}\right) L(s, \omega^{\text{prim}}).$ 
Hence, from \eqref{eqn:abelian-infinity}, \eqref{eqn:ratio-of-dirichlet}, and Theorem \ref{thm:main-theorem}, 
for any integer $m$ with $\frac{k+l-2}{2} < m < m+1 < k$, it follows that:
\begin{equation}
	\frac{L(m, f\times g)}{L(m+1, f \times g)} \ \in \ \mathbb{Q}(f)\mathbb{Q}(g).
\end{equation}

This is proved in greater generality in the case of $L$-functions attached to cohomological cuspidal automorphic representations of 
$({\rm GL}_n \times {\rm GL}_{n'})/F,$ when $nn'$ is even and $F$ a totally real number field, by Harder and the second author \cite{harder-raghuram}.
\medskip

\subsection{An algorithm for the algebraic part of $D(m,f\times g)$}
\label{subsection: algorith-gl2-x-gl2}
\hfill

If $f \in S_k(\Gamma_1(N))$ is a newform then 
$f_\rho = \overline{f(-\overline{z})} = f^\rho$ is again a newform of $S_k(\Gamma_1(N)),$  
where, recall that for any automorphism $\sigma$ of $\C$, one defines 
$f^\sigma := \sum_{n=1}^\infty a(n,h)^\sigma q^n.$  

\medskip 

\noindent \textbf{Input:} $f \in S_k(N, \chi)$ a newform of level $N,$ and $g \in S_l(N)$ an arbitrary cusp form with algebraic Fourier coefficients.
  
\medskip 
   
\noindent \textbf{Output:} The algebraic number  $\frac{D(m, f \times g)}{\pi^k \langle f, f \rangle}$
for $\frac{k+l-2}{2} < m < k.$ 

\medskip 
 
\begin{enumerate}

	\item[Step 1.] Extend $f_\rho$ to a basis consisting of eigenforms $S_k(\Gamma_1(N))$ which we denote by 
	$\mathcal{B} = \{f_0 := f_\rho, f_1, \dots, f_n\}$. (The dimension is $n+1$.)
	
	\smallskip
	
	\item[Step 2.] Fix an $m$ in the critical set. Put $r = k-1-m$.	Pick the correct Eisenstein series $E_{\lambda, N}^*(z,\chi\psi)$ where $\lambda = k-l-2r$. 
	
	\smallskip
	
	\item[Step 3.] Find the holomorphic projection $h_0$ of $g\delta_{\lambda}^{(r)}E_{\lambda,N}(z, \chi\psi)$ using the recursive relations in Lemma \ref{lemma:h-recursive-relations}.
	
	\smallskip
	
	\item[Step 4.] Write the modular form $h_0$ in terms of $\mathcal{B}$, i.e., $h_0 = \alpha_0 f_0 + \alpha_1 f_1 + \cdots + \alpha_n f_n$.
	
	\smallskip
	
	\item[Step 5.]   $ \textbf{Return:} \quad c_r \alpha_0.$  
\end{enumerate}
This is the desired value due to Lemma \ref{lemma: fgdelta_fh0}, Lemma \ref{lemma: zero-lemma}, Theorem \ref{thm: l-value-and-petersson-inner-product}, 
and $\langle f_0, f_i \rangle = 0$ for $i =1 , \dots, n$ and $\langle f^\rho , f^\rho \rangle = \langle f, f \rangle$. 
The constant $c_r$ appears in Theorem \ref{thm: l-value-and-petersson-inner-product}.

\smallskip

This algorithm can be extended to the case when $f$ is only assumed to be an eigenform instead of newform. See Section \ref{section: weight-26-x-13} for an example.

\medskip 
\subsection{An algorithm for the algebraic part of $D(m,f)$}
\label{subsection: algorith-gl2}
We describe a well-known algorithm to calculate the algebraic part of the special value of $L(s,f)$ for the standard $L$-function of a new form 
$f$ at $s = m$ a critical value. 
The term algebraic is explained below. 
We need some preliminaries on computing $L$-values by modular symbols.

\subsubsection{Special values of $L$-functions via modular symbols}
\label{sec:speical-values-of-l-functions-via-modular-symbols}

Let $\Gamma$ be a congruence subgroup of $\SL_2(\mathbb{Z})$ and $\mathbb{P}^1(\mathbb{Q}) = \mathbb{Q}\cup \{\infty\}.$ 
Define $\mathbb{M}_2$ to be the free abelian group on the symbols $\{\alpha,\beta\} \in \mathbb{P}^1(\mathbb{Q})$ modulo the relations 
$ \{\alpha,\beta\} + \{\beta,\gamma\} + \{\gamma, \alpha\}=0 $ for all $\alpha, \beta ,\gamma \in \mathbb{P}^1(\mathbb{Q})$ and modulo all torsion. 
Let $\mathbb{Z}[X,Y]_{n}$ be the abelian group of homogeneous polynomials of degree $n$ in two variables $X$ and $Y$. 
Make $\mathbb{Z}[X,Y]_n$ into a $\Gamma$-module: if $\gamma = \left(\begin{smallmatrix} a & b \\c & d\end{smallmatrix} \right) \in \Gamma$ and 
$P \in \mathbb{Z}[X,Y]_n$ then define $(\gamma P)(X,Y) = P(dX-bY,-cX+aY).$ The abelian group $\mathbb{M}_2$ can also be made into a 
$\Gamma$-module by the action $ g\{\alpha, \beta\} := \{g\alpha, g\beta\}$ for all $g\in \Gamma$ and $\{\alpha, \beta\} \in \mathbb{M}_2$. 
Define the $\Gamma$-module  $\mathbb{M}_k = \mathbb{Z}[X,Y]_{k-2}\otimes_{\mathbb{Z}}\mathbb{M}_2 $
with the $\Gamma$ acting diagonally.
\begin{definition}(Modular Symbols) 
For an integer $k\geq 2$ and a congruence subgroup $\Gamma$, the space $\mathbb{M}_k(\Gamma)$ of weight $k$ modular symbols for 
$\Gamma$ is the quotient of $\mathbb{M}_k$ by all the relations $\gamma x-x$ for $x \in \mathbb{M}_k$, and $\gamma \in G$ and by any torsion.
\end{definition}

For $P\in \mathbb{Z}[X,Y]_{k-2}$ and $\gamma \in \SL_2(\mathbb{Z})$  the associated \textit{Manin symbol} is
\begin{equation*}
	[P, \gamma] := \gamma (P\otimes\{0,\infty\}) \in \mathbb{M}_k(\Gamma). 
\end{equation*}
When $\Gamma= \Gamma_1(N)$ and if $\gamma = \left(\begin{smallmatrix} a & b \\ c & d\end{smallmatrix} \right), \ 
\gamma' = \left(\begin{smallmatrix} a' & b'  \\ c' & d'\end{smallmatrix}\right) \in \Gamma_1(N)$ 
are such that $(c,d)\equiv (c',d')\pmod{N}$, then $[P, \gamma]= [P, \gamma']$. Hence, the Manin symbol $[P, \gamma]$ is determined by 
$P$ and the lower two entries $c,d$ of the matrix $\gamma$. So we take $[P,(c,d)]$ itself to be a Manin symbol. For the following proposition, see, 
for example, Stein \cite[Prop.\,8.3]{stein}.

\begin{proposition}
	The Manin Symbols generate $\mathbb{M}_k(\Gamma).$
\end{proposition}

For any ring $R$ put $\mathbb{M}_k(\Gamma, R):= \mathbb{M}_k(\Gamma) \otimes_\mathbb{Z} R.$
There is a notion of boundary modular symbols and cuspidal symbols denoted by $\mathbb{B}_k(\Gamma)$ and $\mathbb{S}_k(\Gamma)$ respectively. 
One can also define Hecke operators on the space of modular symbols; see Stein \cite[Section 8.3]{stein}. 
The space of cuspidal symbols is stable under the action of Hecke operators.

\begin{theorem}
Let $S_k(\Gamma)$ and $\bar{S}_k(\Gamma)$ denote the space of holomorphic and anti-holomorphic cusp forms respectively. Then the pairing
$( \cdot , \cdot ) : (S_k(\Gamma)\oplus\bar{S}_k(\Gamma)) \times \mathbb{M}_k(\Gamma, \C) \rightarrow \mathbb{C}$
$$
( (f_1,f_2), P\otimes\{\alpha, \beta\} ) \ \mapsto \ \int_{\alpha}^{\beta} f_1(z)P(z,1)dz + \int_{\alpha}^{\beta}f_2(z)P(\bar{z},1)d\bar{z}
$$
 is Hecke equivariant, i.e., for $x \in \mathbb{M}_k(\Gamma, \mathbb{C})$ we have
$(T_n(f_1,f_2),x ) = ( f_1,f_2,T_n(x)).$
	\label{thm: fundamental-properties-of-modular-symbols}
\end{theorem} 

Assume $\Gamma$ is a finite index subgroup of $\SL_2(\Z)$ such that $\eta \Gamma \eta = \Gamma,$ where $\eta = \text{diag}[-1,1].$ 
Then an involution $\iota^*$ on a Manin symbol is defined by $\iota^* [P(X,Y), (u,v)] = [P(X,-Y), (-u,v)].$ The involution $\iota^*$ commutes with the action of 
Hecke operators and it stabilizes the cuspidal subspace. We denote by $\mathbb{S}_k(\Gamma)^-$ and $\mathbb{S}_k(\Gamma)^+$ to be subspace of 
the cuspidal symbols which are the $-1$ and $+1$ eigenspaces of the action of $\iota^*,$ respectively. 
	\begin{theorem}
	The integration pairing $( \cdot , \cdot )$ induces nondegenerate Hecke-equivariant bilinear pairings
	\begin{equation*}
		{S}_k(\Gamma) \times \mathbb{S}_k(\Gamma, \mathbb{C})^- \rightarrow \mathbb{C}    \quad \text{and} \quad 
		\overline{S}_k(\Gamma)  \times \mathbb{S}_k(\Gamma, \mathbb{C})^+  \rightarrow \mathbb{C}.
	\end{equation*}
\end{theorem}
The signs $\pm$ are interchanged in the above theorem  in comparison with the one given in \cite{stein} 
because we do not take the `$-$' sign in the definition of $\iota^*.$ This convention is adopted because it is the one implemented in SAGE \cite{sage} 
which we use for the computations that follow.

For a cusp form $f \in S_k(N)$ with Fourier expansion $\sum a(n,f) q^n$, the associated Dirichlet series $D(s, f)$ and completed $L$-function $L(s,f)$ are given by: 
\begin{equation}
D(s,f) := \sum_{n=1}^\infty \frac{a(n, f)}{n^s}, \quad 
L(s, f) :=  (2\pi)^{-s} \Gamma(s) D(s, f) =  
\int_0^\infty f(\sqrt{-1}t)t^s \frac{dt}{t}, \quad \Re(s) \gg 0.
\label{eqn:sp-value-appendix} 
\end{equation}
Consider the modular symbols $X^{j}Y^{k-2-j}\otimes\{0,\infty\}$ for $j=0,1,2,\dots, k-2$. 
The non-degenerate and Hecke equivariant pairing of these modular symbols against the cusp form $f$ is:
\begin{equation*}
	(f, X^{j}Y^{k-2-j}\otimes \{0,\infty\} ) = \int_0^{i \infty} f(z) z^j dz = i^{j+1} \int_{0}^{\infty} f(i t) t^{j} dt.
\end{equation*}
By \eqref{eqn:sp-value-appendix}, for $m=1,2\dots,k-1$, we have:
\begin{align}
	D(m, f) = \frac{(-2\pi \sqrt{-1})^{m}}{(m-1)!} \cdot ( f, X^{m-1}Y^{k-2-(m-1)}\otimes\{0,\infty\}).\label{eqn:modsym-spvalue-appendix}
\end{align}
For $m=1,2\dots,k-1$, define the algebraic part of a critical value $L(m,f)$ by: 
\begin{align}
	\frac{D(m,f)}{(-2\pi \sqrt{-1})^{m-1} D(1,f)} & 
	= \frac{1}{(m-1)!}\cdot \frac{(f, X^{m-1}Y^{k-2-(m-1)}\otimes\{0,\infty\})}{(f, Y^{k-2}\otimes \{0, \infty\})} \;\; \text{for }m =1,3, \dots, \label{eqn: odd-l-value}\\
	\frac{D(m,f)}{(-2\pi \sqrt{-1})^{m-2} D(2,f)} & 
	= \frac{1}{(m-1)!}\cdot \frac{(f, X^{m-1}Y^{k-2-(m-1)}\otimes\{0,\infty\})}{(f, X^{1}Y^{k-3}\otimes \{0, \infty\})} \;\; \text{for }m =2,4, \dots. \label{eqn: even-l-value}
\end{align}

\medskip
\subsubsection{Some preliminaries for $\SL_2(\mathbb{Z})$}

The following are well known: 

\smallskip
\begin{enumerate}
\item The space of cusp forms $S_k(\SL_2(\Z))$ has co-dimension $1$ in $M_k(\SL_2(\Z))$ and let $d$ denote the dimension of $S_k(\SL_2(\Z))$. 

\smallskip
\item The space $\mathbb{M}_k(\SL_2(\Z),\Q)^-$ (resp., $\mathbb{M}_k(\SL_2(\Z), \Q)^+$) is generated by $X^iY^{k-2-i} \otimes \{ 0, \infty\}$, 
or equivalently by the Manin symbols $[X^iY^{k-2-i},(0,1)]$ for odd $i$ (resp., for even $i$). For brevity, we denote the Manin symbol $[X^iY^{k-2-i},(0,1)]$ by $(i,0,1)$.

\smallskip
\item $\mathbb{S}_k(\SL_2(\Z), \Q)^- = \mathbb{M}_k(\SL_2(\Z), \Q)^- .$  

\smallskip
\item $\mathbb{S}_k(\SL_2(\Z), \Q)^+$ has co-dimension $1$ in $\mathbb{M}_k(\SL_2(\Z), \Q)^+.$
\end{enumerate}

\medskip
\subsubsection{} {\it The algorithm for even critical points.}
\label{subsubsection: algorithm-gl2-even} \\

\noindent \textbf{Input:} The first $d$ Fourier coefficients $a(1,f) = 1, \, a(2,f),\dots, \, a(d, f)$ 
of a normalized eigenform $f \in S_k(\SL_2(\mathbb{Z}))$ which are assumed to be real.

\medskip

\noindent \textbf{Output:} $\frac{D(m,f)}{(-2\pi \sqrt{-1})^{m-2} D(2,f)}$ for $m=2,4, \dots, k-2.$

\smallskip
\begin{enumerate}
	\item[Step 1.] Find a basis $b_1^-, b_2^-, \dots, b_d^-$ for the space $\mathbb{M}_k(\SL_2(\Z), \mathbb{Q})^- = \mathbb{S}_k(\SL_2(\Z), \mathbb{Q})^-.$
	
	\smallskip
	\item[Step 2.]  Express the Manin symbols $(i,0,1)$ for $i$ odd and $0\leq i \leq k-3$ in terms of the basis $b_1^-, b_2^-,\dots, b_d^-.$ 
	Let $A^- = [a^-_{ij}]$ be the matrix with rational entries such that $(i,0,1) = \sum_{j=1}^d a_{ji}^-b_j^-.$
	
	\smallskip
	\item[Step 3.] Compute the matrices $ M_2^-, \dots, M_d^-$ for the action of the Hecke operators $T_2^-, \dots, T_d^-$ 
	on $\mathbb{M}_k(\SL_2(\Z), \mathbb{Q})^-$ with respect to the basis $b_1^-, b_2^-, \dots, b_d^-.$ 
	
	\smallskip
	\item[Step 4.] Form the matrix
	\[	\mathcal{L}^- = \left( \begin{array}{c}  {^t}M_2^- - a(2,f)I_{d} \\
		 \vdots \\
		 {^t}M_d^- - a(d,f) I_{d}
	\end{array}\right). \]
	
	\smallskip
	\item[Step 5.] Row-reduce the matrix $\mathcal{L}^-$ and pick any non-zero vector 
	$w^- := {^t}(w_1^-, w_2^-, w_3^-, \dots, w_d^-)$ from the null space of $\mathcal{L}^-$. 
	
	\smallskip
		\item[Step 6.] Compute the $\frac{k-2}{2} \times 1$ column matrix $E :=  {^t}A^- \cdot w^-.$ Let $c_1,c_3, \dots, c_{k-3}$ be the entries of $E$.

	\smallskip 	
	\item[Step 7.] $\textbf{Return:} \qquad \dfrac{1}{(i-1)!}\cdot \dfrac{c_i}{c_1} \text{ for $i=1,3,\dots, k-3$}.$
\end{enumerate}

This is the desired output by \eqref{eqn: even-l-value} and 
\begin{gather*}
	\frac{(f, X^{i-1} Y^{k-2-(i-1)}\otimes \{0,\infty\})}{(f, Y^{k-2}\otimes \{0,\infty\} )} = \frac{c_i}{c_1} \:\:\: \mbox{for}\:\: i=1,3, \dots k-3.
\end{gather*}
See \ref{section: correctness-of-the-algorithms} for a proof of the algorithm.

\medskip
\subsubsection{} {\it The algorithm for odd critical points}
\label{subsubsection: algorithm-gl2-odd} \\

\noindent \textbf{Input:} The first $d$ Fourier coefficients $a(1,f) = 1, \, a(2,f),\dots, \, a(d, f)$ 
	of a normalized eigenform $f \in S_k(\SL_2(\mathbb{Z}))$ which are assumed to be real.

\medskip

\noindent \textbf{Output:} $\frac{D(m,f)}{(-2\pi \sqrt{-1})^{m-1} D(1,f)}$ for $m=1,3, \dots,k-1.$

\medskip
\begin{enumerate}
	\item[Step 1.] Find a basis $b_1^+, b_2^+, \dots, b_{d+1}^+$ for the space $\mathbb{M}_k(\SL_2(\Z), \mathbb{Q})^+.$
	
	\smallskip
	\item[Step 2.]  Express the Manin symbols $(i,0,1)$ for $i$ even and $0\leq i \leq k-2$ in terms of the basis $b_1^+, b_2^+, \dots, b_{d+1}^+.$ 
	Let $A^+ = [a^+_{ij}]$ be a matrix with rational entries such that  $(i,0,1) = \sum_{j=1}^{d+1} a^+_{ji}b_j^+.$
	
	\smallskip
	\item[Step 3.] Compute the matrices $ M_2^+, \dots, M_{d}^+$ of the Hecke operators $T_2^+, \dots, T_{d}^+$ acting on 
	$\mathbb{M}_k(\SL_2(\Z), \mathbb{Q})^+$ with respect to $b_1^+, b_2^+, \dots, b_{d+1}^+.$
	
	\smallskip
	\item[Step 4.] Form the block matrix
	\[	\mathcal{L}^+ = \left( \begin{array}{c}  {^t}M_2^+ - a(2,f) I_{d+1} \\
		 \vdots \\
		 {^t}M_{d}^+ - a(d,f) I_{d+1}
	\end{array}\right). \]
	
	\smallskip
	\item[Step 5.]Row-reduce the matrix $\mathcal{L}^+$ and pick any non-zero vector $w^+ := {^t}(w_1^+, w_2^+, w_3^+, \dots, w_{d+1}^+)$ 
	from the null space of $\mathcal{L}^+$.
	
	\smallskip
	\item[Step 6.] Compute the $ \frac{k}{2}\times 1$ column matrix $O :=  {^t}A^+ \cdot w^+.$ Let $d_0,d_2, \dots, d_{k-2}$ be the entries of $O$.
	
	\smallskip
	\item[Step 7.] $\textbf{Return:} \qquad \dfrac{1}{(i-1)!}\cdot \dfrac{d_i}{d_0} \text{ for $i=0,2, \dots, k-2$}.$	
\end{enumerate}
This is the desired output by \eqref{eqn: odd-l-value} and
\begin{gather*}
	\frac{(f, X^{i-1} Y^{k-2-(i-1)}\otimes \{0,\infty\})}{(f, Y^{k-3}\otimes \{0,\infty\} )} = \frac{d_i}{d_0} \:\:\: \mbox{for}\:\: i=0,2, \dots, k-2.
\end{gather*}

\medskip
\subsubsection{The correctness of the algorithms} \label{section: correctness-of-the-algorithms}
The correctness of both the algorithms, after a slight modification, are the same as in 
Manin \cite[Sections 4.1, 4.2, and 6]{manin}. For the reader's convenience a proof of the algorithm is adumbrated for even critical points.

Let $\widetilde{\mathbb{S}}_k(\SL_2(\Z), \C)^-$ denote the $\C$-dual space of $\mathbb{S}_k(\SL_2(\Z), \C)^-.$ 
	Define the Hecke action on the dual space $\widetilde{\mathbb{S}}_k(\SL_2(\Z), \C)^-$ in the usual way, i.e., 
	the $n$-th Hecke operator $\tilde{T}^-_n$ acts on $\tilde{f} \in \widetilde{\mathbb{S}}_k(\SL_2(\Z), \C)^-$ by 
	$\langle\!\langle f, \tilde{T}^-_n(\tilde{f})\rangle\! \rangle = \langle\!\langle T^-_n(f), \tilde{f} \rangle\!\rangle $, where $ \langle\!\langle f, \tilde{f} \rangle\!\rangle$ is evaluation of $\tilde{f}$ on $f.$
	
\begin{proof}(For even critical points.)
	 Since the $T_n$'s are a commuting family of diagonalizable operators, the space $S_k(\SL_2(\Z))$ can be decomposed into 
	 Hecke isotypic subspaces which are all one-dimensional by multiplicity one. 
	 Since the pairing $(\cdot , \cdot)$ is Hecke equivariant, the same is true for the spaces 
	 $\mathbb{S}_k(\SL_2(\Z), \C)^-$ and $\widetilde{\mathbb{S}}_k(\SL_2(\Z), \C)^-$. 
	 We can write 
	 \begin{equation*}
	 	\widetilde{\mathbb{S}}_k(\SL_2(\Z), \C)^- = \bigoplus_{f \in S_k(\SL_2(\Z))} \C \omega_f
	 \end{equation*}
	 for Hecke isotypic decomposition where $\omega_f(\_) = (f,\_) \in \widetilde{\mathbb{S}}_k(\SL_2(\Z), \C)^-$ has the same system of eigenvalues as $f$. 
	 The space $\mathbb{S}_k(\SL_2(\Z), \C)^-$ has a rational subspace 
	 $\mathbb{S}_k(\SL_2(\Z), \Q)^-$ which is stable under the Hecke operators. Let $b_1^-, b_2^-, \dots, b_d^-$ be a basis of $\mathbb{S}_k(\SL_2(\Z), \Q)^-$ 
	 and $\tilde{b}_1^-, \tilde{b}_2^-, \dots, \tilde{b}_d^-$ be the dual basis. Let $M_n^-$ be the $d\times d$ matrix of $T_n$ with respect to this basis. 
	 Then with respect to the dual basis $\tilde{b}_1^-, \tilde{b}_2^-, \dots, \tilde{b}_d^-$ the matrix of $\tilde{T}_n^-$ is simply ${^t}M_n^-$.
	 
	 A vector in the isotypic subspace $\C \omega_f$ in $\widetilde{\mathbb{S}}_k(\SL_2(\Z), \Q)^-$ 
	 lies in $\bigcap_{n } \text{Null} ({^t}M_n^- - a(n,f)1\!\!1_d).$ Since $f$ is completely determined the first $d$ many Fourier coefficients, 
	 the intersection can be taken over a smaller set
	 \[ \bigcap_{n=1 }^d \text{Null}({^t}M_n^- - a(n,f)1\!\!1_d).\]
	 (This is \textit{Step 5} in the \textit{Algorithm} where one row-reduces the matrix $\mathcal{L}^-$ to pick a vector from its null space.) 
	 Let $\tilde{w}^- = {^t}(w_1^-, w_2^-, \dots, w_n^-) = \sum_{i=1}^d w_i^- \tilde{b}_i^-$ be a non-zero vector from the intersection. 
	 By construction $\tilde{w}^-$ lies in the isotypic subspace  of $\widetilde{\mathbb{S}}_k(\SL_2(\Z), \C)^-$ determined by
	 $\tilde{T}_n^- \tilde{w}^- = a(n,f) \tilde{w}^-.$ By multiplicity one, this is the same as $\C \omega_f$. 
	 Hence, 
	 \begin{equation} 
	 \tilde{w}^- = \lambda \omega_f, \quad \mbox{for some $\lambda \in \C^\times$}. 
	 \label{eqn: equality-of-functionals}\end{equation}
	  For an odd integer $i$ and $0 \leq i \leq k-2,$ if 
	  $X^iY^{k-2-i}\otimes \{0, \infty\} = \sum_{j=1}^d a_{ji}^- b_j^{-}$ with $a_{ij} \in \mathbb{Q}$ then
	 $$
	 \tilde{w}^-(X^iY^{k-2-i}\otimes \{0, \infty\}) = \tilde{w}^-(\sum_{j=1}^d a_{ji}^- b_j^{-}) = \sum_{j=1}^d a_{ji}^- \tilde{w}^-(b_j^-) = \sum_{j=1}^d  a_{ji}^- w_j^-.
	 $$
	 Set $E:=  {^t}A^- \cdot w^-$, where $A^- = [a_{ij}^-]$, then from \eqref{eqn: equality-of-functionals} it is clear that
	 \begin{equation*}
	 	{^t}E = \left( \lambda (f, X^1Y^{k-3}\otimes \{0, \infty\}),\,\, \lambda(f, X^3Y^{k-5}\otimes \{0, \infty\}), \dots, \,\, 
		\lambda(f , X^{k-3}Y^{1}\otimes \{ 0, \infty\})\right)
	 \end{equation*}
	 which proves the given algorithm.
	\end{proof}


\subsection{Sturm's criterion for congruence between cusp forms}
To check congruence between cusp forms we will use the following criterion of Sturm \cite{sturm}; see also Stein \cite[p.\,173]{stein}.

\begin{theorem}[Sturm]
	Let $\mathfrak{P}$ be a prime ideal of a number field $K$. Let $f,g \in M_k(\Gamma, \mathcal{O}_K)$ be modular forms of weight $k$ for 
	a congruence subgroup $\Gamma$ of level $N$ with Fourier coefficients in the ring of integers $\mathcal{O}_K$ of $K$. Let $m = [\SL_2(\mathbb{Z}):\Gamma].$
	Suppose 
	\begin{equation*}
		a(n,f) \equiv a(n,g) \pmod{\mathfrak{P}}
	\end{equation*}
	for all
	\begin{equation*}
		n \leq \begin{cases}
			\frac{km}{12}- \frac{m-1}{N} & \mbox{if $f-g \in S_k(\Gamma, \mathcal{O}_K$}), \\
			\frac{km}{12} & \mbox{otherwise.}
		\end{cases}
	\end{equation*}
	Then $f \equiv g \pmod{\mathfrak{P}}$. \label{corollary: sturm-bound}
\end{theorem}
The bound for $n$ in the above theorem is called the Sturm bound for $\Gamma$ and weight $k$.

\medskip

\section{Examples}

\subsection{$S_{24}(\SL_2(\mathbb{Z}))\times S_{12}(\SL_2(\mathbb{Z}))$}
\label{section: weight-24-x-12}

 The dimension of  $S_{24}(\SL_2(\mathbb{Z}))$ is $2$. There are no old forms as the level is $1$. The two distinct newforms, say, ${f_1}$ and ${f_2}$ have 
 Fourier expansions:
\begin{align*}
		f_1 &= q + (\beta_{0})q^2 + O(q^3),\\
		f_2 &= q + (-\beta_{0} + 1080)q^2 + O(q^3),
\end{align*}
where $\beta_0 = 12 \sqrt{144169}$. The coefficients are in the number field $\mathbb{Q}(\beta_0).$ 
The prime number $144169$ is ramified in the number field $\mathbb{Q}(\beta_0)$. 
Let the prime ideal lying above $144169$ be $\mathfrak{P}_{144169}$. 
Also, $f_{1\rho} = {f_1}^\rho = {f_1}$ and $f_{2\rho} = {f_2}^\rho = {f_2}$ as the coefficients are totally real.
Their difference is
\begin{align*}
f_1 -f_2 = (2\beta_0 - 1080)q^2 + O(q^3) .
\end{align*}
The ideal factorization of $(2\beta_0 - 1080)$ in the number field $\mathbb{Q}(\beta_0)$ is
\begin{multline*}
 (\left(2, \frac{1}{24} \beta_{0} - 23\right))^{3} \cdot (\left(2, \frac{1}{24} \beta_{0} - 22\right))^{3} \cdot (\left(3, \frac{1}{24} \beta_{0} - 23\right)) \\ \cdot (\left(3, \frac{1}{24} \beta_{0} - 22\right)) \cdot (\left(144169, \frac{1}{24} \beta_{0} + 72062\right)).
\end{multline*}
Therefore, since the Sturm bound is $2$ for $S_{24}(\SL_2(\mathbb{Z}))$, by Theorem \ref{corollary: sturm-bound} one has: 
\begin{equation*}
		f_1 \equiv f_2\pmod{\mathfrak{P}_{144169}}.
\end{equation*}

\medskip

 The space $S_l(\Gamma_0(N)) = S_{12}(\SL_2(\mathbb{Z}))$ is of dimension $1$ 
 and generated by the unique Ramanujan cusp form $\Delta$. Put $g = \Delta.$ The Fourier expansion of $g$ is $q - 24q^2 +O(q^3)$. 
 
 \medskip

\begin{theorem} 
\label{thm:sl4zxsl2z}
For $f_1,f_2 \in S_{24}(\SL_2(\mathbb{Z}))$ and $g \in S_{12}(\SL_2(\mathbb{Z}))$ as above, and any integer $m$ with  
$18 \leq m \leq 22$, one has:
	\begin{align*}
		\frac{{L}(m,f_1\times g)}{{L}(m+1,f_1\times g)}\equiv \frac{{L}(m,f_2 \times g)}{{L}(m+1,f_2\times g)} \pmod{ \mathfrak{P}_{144169}}. 
	\end{align*}
\end{theorem}
This is verified using the algorithm given in Section \ref{subsection: algorith-gl2-x-gl2}.

\medskip
\subsubsection{Values of $D(23,f_1\times g)$ and $D(23,f_2\times g)$}

As \textit{Step 1}, fix the basis $\mathcal{B} = \{ f_1, f_2\}$ of $S_{24}(\SL_2(\mathbb{Z}))$ consisting of newforms. Proceed to \textit{Step 2}.
The critical point is $23.$ So put $r=0$ and $\lambda = k - l - 2r = 12$. The Eisenstein series to be considered is $E_{\lambda, N}^*(z,1\!\!1) = E_{12,1}^*(z,1\!\!1).$  The space of Eisenstein series of weight $12$ and trivial character is of dimension $1$. The $q$-expansion of $E_{12,1}^*$ is
\begin{align*}
		E_{12,1}^*(z,1\!\!1) &= 1 + \frac{65520}{691} q + \frac{134250480}{691} q^{2} + \frac{11606736960}{691} q^{3} + O(q^{4}).
    \end{align*}
Furthermore, 
    \begin{equation*}
		g.\delta_{12}^{(0)}E_{12,1}^*(z,1\!\!1) = g.E_{12,1}^*(z,1\!\!1) = q + \frac{48936}{691} q^{2} + \frac{132852132}{691} q^{3} + O(q^{4}).
	\end{equation*}
Also, $g.E_{12,1}^*(z,1\!\!1)$ is a modular form (in fact, a cusp form) of weight $24$. So there is no need to take the holomorphic projection. 
Proceed to $\textit{Step 4}.$ In terms of the basis $\mathcal{B} = \{ f_1, f_2\}$
\begin{equation*}
		g.E_{12,1}^*(z,1\!\!1) = \left(-\frac{27017}{2390898696} \beta_{0} + \frac{50418272}{99620779}\right)f_1 + \left(\frac{27017}{2390898696} \beta_{0} + \frac{49202507}{99620779}\right)f_2 .
\end{equation*}
For \textit{Step 5} one gets
\begin{equation*}
		\frac{D(23,f_1\times g)}{\pi^{24} \langle f_1, f_1 \rangle} = c_0\frac{\langle {f_1}^\rho, g\delta_{\lambda}^{(r)}E_{\lambda, N}^* \rangle}{\pi^{24} \langle f_1, f_1 \rangle} =c_0 \left(-\frac{27017}{2390898696} \beta_{0} + \frac{50418272}{99620779}\right).
\end{equation*}
Similarly, for the pair $(f_2,g)$ one gets
\begin{equation*}
	\frac{D(23,f_2\times g)}{\pi^{24} \langle f_2, f_2 \rangle} = c_0\frac{\langle {f_2}^\rho, g\delta_{\lambda}^{(r)}E_{\lambda, N}^* \rangle}{\pi^{24} \langle f_2, f_2 \rangle}                       =c_0\left(\frac{27017}{2390898696} \beta_{0} + \frac{49202507}{99620779}\right) .
\end{equation*}

\medskip
\subsubsection{Values of $D(22,f_1\times g)$ and $D(22,f_2\times g)$}

As the basis is already fixed proceed to $\textit{Step 2}.$
For the critical point $22$ put $r=1$ and so $\lambda = k-l-2r = 10$. So the Eisenstein Series is $E_{\lambda, N}^*(z,1\!\!1) =  E_{10,1}^*(z,1\!\!1) $ and its $q$-expansion is
\begin{align*}
		E_{10,1}^*(z,1\!\!1) = 1 - 264 q - 135432 q^{2} - 5196576 q^{3} - 69341448 q^{4} - 515625264 q^{5} + O(q^{6}) .
\end{align*}
Proceeding to \textit{Step 4}, by Lemma \ref{lemma: hol-projection}, there are unique $h_0$ and $h_1$ sastisfying
\begin{align*}
		g\delta_{10}^{(1)}E_{10,1}^*(z,1\!\!1) = h_0 + \delta_{22}^{(1)}h_1 .
\end{align*}
Using the recursive relations in Lemma \ref{lemma:h-recursive-relations} calculate the holomorphic projection $h_0:$ 
\begin{align*}
		h_0 &= -\frac{5}{11} q - \frac{24}{11} q^{2} - \frac{977148}{11} q^{3} + O(q^{4})\\
        &= \left(\frac{223}{38060616} \beta_{0} - \frac{365440}{1585859}\right) f_1 + \left(-\frac{223}{38060616} \beta_{0} - \frac{355405}{1585859}\right) f_2 .
\end{align*}
Proceeding to the last \textit{Step 6}, 
\begin{align*}
\frac{D(22, f_1 \times g)}{\pi^{24}\langle f_1, f_1 \rangle} = \frac{\langle f_1^\rho, \delta_{10}^{(1)}E_{10,1}^* \rangle }{\pi^{24}\langle f_1, f_1 \rangle} = \frac{\langle f_1, h_0\rangle}{\pi^{24}\langle f_1, f_1 \rangle} = c_1 \left(\frac{223}{38060616} \beta_{0} - \frac{365440}{1585859}\right).
\end{align*}
Similarly,
\begin{equation*}
	\frac{D(22, f_2 \times g)}{\pi^{24}\langle f_2, f_2 \rangle} = c_1 \left(-\frac{223}{38060616} \beta_{0} - \frac{355405}{1585859}\right).
\end{equation*}

\medskip
\subsubsection{Ratios of successive critical values}
Since $N=1$ and $\chi$ and $\psi$ are trivial, we have $L_N(s, \chi \psi) = \zeta(s)$, the Riemann Zeta function. It is well known that: 
$$\zeta(2 \cdot 23-k-l) = \zeta(10) = \frac{691}{638512875} \, \pi^{12}, \quad \zeta(2\cdot 22 +2  - k - l ) = \zeta(8) = \frac{1}{93555} \, \pi^{10}.$$ 
Therefore, for $i=1,2$
\begin{equation*}
	\frac{L_\infty(22, f_i\times g)\zeta(8)}{L_\infty(23, f_i \times g) \zeta(10)} =  \frac{13650}{83611}.
	\end{equation*}
The ratios of the special values of the completed $L$-functions are
\begin{equation*}
	\frac{L(22, f_1\times g)}{L(23, f_1\times g)} = -\frac{1}{1905750} \beta_{0} + \frac{51866}{317625} \quad \text{ and } \quad 
	\frac{L(22,f_1 \times g)}{L(23, f_2 \times g)} = \frac{1}{1905750} \beta_{0} + \frac{51686}{317625}.
	\end{equation*}
The ideal factorization of the quantity $\frac{L(22,f_1\times g)}{L(23,f_1\times g)} - \frac{L(22,f_2\times g)}{L(23,f_2\times g)}$ in the number field $\mathbb{Q}(\beta_0)$ is
\begin{multline*}
(\left(2, \frac{1}{24} \beta_{0} - 23\right))^{2} \cdot (\left(2, \frac{1}{24} \beta_{0} - 22\right))^{2} \cdot (\left(3, \frac{1}{24} \beta_{0} - 23\right))^{-1} \cdot (\left(3, \frac{1}{24} \beta_{0} - 22\right))^{-1} \cdot \\ (\left(5, \frac{1}{24} \beta_{0} - 21\right))^{-3} \cdot (\left(5, \frac{1}{24} \beta_{0} - 19\right))^{-3} \cdot (\left(7, \frac{1}{24} \beta_{0} - 20\right))^{-1} \cdot (\left(7, \frac{1}{24} \beta_{0} - 18\right))^{-1} \cdot \\ (\left(11, \frac{1}{24} \beta_{0} - 20\right))^{-2} \cdot (\left(11, \frac{1}{24} \beta_{0} - 14\right))^{-2} \cdot (\left(144169, \frac{1}{24} \beta_{0} + 72062\right)).
\end{multline*}
Hence 
\begin{align*}
\frac{L(22,f_1\times g)}{L(23,f_1\times g)} \equiv \frac{L(22,f_2\times g)}{L(23,f_2\times g)} \pmod{\mathfrak{P}_{144169}}.
\end{align*}

\subsubsection{Other ratios}
For $m = 21,20,19$, the ideal factorizations of 
$$\frac{{L}(m,f_1\times g)}{{L}(m+1,f_1\times g)} - \frac{{L}(m,f_2\times g)}{{L}(m+1,f_2\times g)}
$$ 
in the number field $\mathbb{Q}(\beta_0)$ are stated below. The method of computation is along the same lines as in the previous sub-sections.

\begin{multline*}
\frac{L(21,f_1\times g)}{L(22,f_1\times g)} - \frac{L(21,f_2\times g)}{L(22,f_2\times g)} \\= 
(\left(2, \frac{1}{24} \beta_{0} - 23\right))^{-1} \cdot (\left(2, \frac{1}{24} \beta_{0} - 22\right))^{-1} \cdot (\left(3, \frac{1}{24} \beta_{0} - 23\right))^{-3} \cdot (\left(3, \frac{1}{24} \beta_{0} - 22\right))^{-3} \cdot \\ (\left(5, \frac{1}{24} \beta_{0} - 21\right))^{-1} \cdot  (\left(5, \frac{1}{24} \beta_{0} - 19\right))^{-1} \cdot  (\left(7, \frac{1}{24} \beta_{0} - 20\right))^{-1} \cdot (\left(7, \frac{1}{24} \beta_{0} - 18\right))^{-1} \cdot \\ (\left(17, \frac{1}{24} \beta_{0} - 21\right))^{-1} \cdot  (\left(17, \frac{1}{24} \beta_{0} - 7\right))^{-1} \cdot (\left(144169, \frac{1}{24} \beta_{0} + 72062\right)),
\end{multline*}

\begin{multline*}
\frac{L(20,f_1\times g)}{L(21,f_1\times g)} - \frac{L(20,f_2\times g)}{L(21,f_2\times g)} \\=(\left(2, \frac{1}{24} \beta_{0} - 23\right))^{-3} \cdot (\left(2, \frac{1}{24} \beta_{0} - 22\right))^{-3} \cdot (\left(3, \frac{1}{24} \beta_{0} - 23\right))^{-1} \cdot (\left(3, \frac{1}{24} \beta_{0} - 22\right))^{-1} \cdot \\ (\left(5, \frac{1}{24} \beta_{0} - 21\right))^{-1} \cdot (\left(5, \frac{1}{24} \beta_{0} - 19\right))^{-1} \cdot (\left(103, \frac{1}{24} \beta_{0} + 18\right))^{-1} \cdot  (\left(103, \frac{1}{24} \beta_{0} + 40\right))^{-1} \cdot \\ (\left(144169, \frac{1}{24} \beta_{0} + 72062\right)) ,
\end{multline*}
and, 
\begin{multline*}
\frac{L(19,f_1\times g)}{L(20,f_1\times g)} - \frac{L(19,f_2\times g)}{L(20,f_2\times g)} \\= (\left(2, \frac{1}{24} \beta_{0} - 23\right))^{-1} \cdot (\left(2, \frac{1}{24} \beta_{0} - 22\right))^{-1} \cdot (\left(7, \frac{1}{24} \beta_{0} - 20\right))^{-1} \cdot  (\left(7, \frac{1}{24} \beta_{0} - 18\right))^{-1} \cdot \\ (\left(17, \frac{1}{24} \beta_{0} - 21\right))^{-1} \cdot (\left(17, \frac{1}{24} \beta_{0} - 7\right))^{-1} \cdot (\left(19, \frac{1}{24} \beta_{0} - 15\right))^{-1} \cdot  (\left(19, \frac{1}{24} \beta_{0} - 11\right))^{-1} \cdot \\ (\left(144169, \frac{1}{24} \beta_{0} + 72062\right)).
\end{multline*}

\subsubsection[]{The critical point $m=18$}
When the critical point under consideration is $m=18,$
the Eisenstein series is no longer holomorphic.  So $r= 5$ and $\lambda = 24 -12 - 2r = 2$; 
the characters of $f,f'$ and $g$ are all trivial. The correct Eisenstein series is a non-holomorphic one:
\begin{equation*}
	E_{2,1}^*(z, 1\!\!1) = \frac{-12}{4 \pi y} + 1 -24 q - 72 q ^2 - \cdots.
\end{equation*}
See Miyake \cite[Chapter 7]{miyake}. By using  Lemma \ref{lemma:h-recursive-relations} the holomorphic projection $h_0$ is
\begin{align*}
	h_0 &= -\frac{5}{24871} q + \frac{10536}{24871} q^{2} - \frac{212004}{3553} q^{3} + O(q^{4})\\
	&= \left( \frac{1103}{86055052776} \beta_{0} - \frac{385240}{3585627199}\right)f_1 + \left( -\frac{1103}{86055052776} \beta_{0} - \frac{335605}{3585627199}\right)f_2 .
\end{align*}
A similar calculation like in the previous sub-sections yield the ideal factorization of
\begin{multline*}
	\frac{L(18,f_1\times g)}{{L}(19,f_1\times g)} - \frac{L(18,f_2\times g)}{L(19,f_2\times g)} \\= (\left(3, \frac{1}{24} \beta_{0} - 23\right))^{-1} \cdot (\left(3, \frac{1}{24} \beta_{0} - 22\right))^{-1} \cdot (\left(5, \frac{1}{24} \beta_{0} - 21\right))^{-1} \cdot  (\left(5, \frac{1}{24} \beta_{0} - 19\right))^{-1} \cdot \\ (\left(7, \frac{1}{24} \beta_{0} - 20\right))^{-1} \cdot  (\left(7, \frac{1}{24} \beta_{0} - 18\right))^{-1} \cdot  (\left(17, \frac{1}{24} \beta_{0} - 21\right))^{-1} \cdot  (\left(17, \frac{1}{24} \beta_{0} - 7\right))^{-1} \cdot \\ (\left(144169, \frac{1}{24} \beta_{0} + 72062\right)) .
\end{multline*}
This completes the verification of Thm.\,\ref{thm:sl4zxsl2z} for all successive ratios of critical points to the right of the line of symmetry.

\medskip
\subsection{$S_{30}(\SL_2(\mathbb{Z}))\times S_{12}(\SL_2(\mathbb{Z}))$}
\label{section: weight-30-x-12}

The space $S_{30}(\SL_2(\mathbb{Z}))$ is again of dimension $2$. 
Call the newforms $f_1$ and $f_2$. The form $f_1$ is a Galois conjugte of $f_2$. 
The Fourier coefficients of $f_1$ and $f_2$ lie in the number field $\mathbb{Q}(\beta_1)$ where $\beta_1$ satisfies the polynomial $x^2 - 8640 x - 454569984.$ 
The prime $51349$ is ramified in $\mathbb{Q}(\beta_1)$. Let $\mathfrak{P}_{51349}$ be the prime ideal lying above $51349$. 
Like in \ref{section: weight-24-x-12}, one verifies that $f_1 \equiv f_2 \pmod{\mathfrak{P}_{51349}}.$
Take $g = \Delta$ to be Ramanuajan cusp form which is of weight $12$ and level $1$. Then exactly like in the 
\ref{section: weight-24-x-12} we get:

\begin{theorem} For $f_1,f_2 \in S_{30}(\SL_2(\mathbb{Z}))$ and $g \in S_{12}(\SL_2(\mathbb{Z}))$ as above, and for an integer $m$ with 
$21 \leq m \leq 28$, one has:
	\begin{equation*}
		\frac{L(m,f_1 \times g)}{L(m+1, f_1 \times g)} \equiv \frac{L(m,f_2 \times g)}{L(m+1, f_2\times g)} \pmod{ \mathfrak{P}_{51349}}. 
	\end{equation*}
\end{theorem}

\medskip
\subsection{$S_{13}(\Gamma_0(3),\chi)\times S_{6}(\Gamma_0(3))$}
\label{section: weight-13-x-6}

Let $\chi$ be the unique quadratic character modulo $3$. All the eigenforms in  $S_{13}(\Gamma_1(3)) = S_{13}(\Gamma_0(3),\chi)$ are newforms. The $q$-expansions of the  newforms are 
\begin{align*}
		f_1 &= q + 729 q^3+ (4096) q^4 + O(q^5), \\
		f_2 &= q + \nu_1 q^2 + (-3 \nu_{1} - 675)q^3 + (-4328)q^4+  O(q^5),\\
		f_3 &= q + -\nu_1 q^2 + (3 \nu_{1} - 675)q^3 + (-4328) q^4+ O(q^5),
\end{align*}
where $\nu_1 = \sqrt{-8424}$. It can checked that $f_2$ and $f_3$ are Galois conjugates of each other, and in fact,  
$f_{2\rho} = f_2^\rho = f_3$ and $f_{3\rho} = f_3^\rho = f_2$ as the coefficients are in the imaginary quadratic field $\Q(\nu_1).$
However, $f_1$ and $f_2$ are \textit{not} Galois conjugate of each other, since the coefficients of $f_1$ are in $\Q$; also, 
$f_{1\rho} = f_1^\rho = f_1,$  
Using the Sturm bound (Thm.\,\ref{corollary: sturm-bound}) one verifies 
\begin{equation*}
f_1 \equiv f_2 \pmod{\mathfrak{P}_{13}},
\end{equation*}
where $\mathfrak{P}_{13} = (\left(13, \frac{1}{18} \nu_1\right))$ is the prime ideal lying above $(13) \subset \mathbb{Z}$ in the number field $\mathbb{Q}(\nu_1).$

\medskip

The space $S_6(\Gamma_0(3))$ is one-dimensional, spanned by $g$ whose fourier expansion is 
$$
g = q - 6q^2 + 9q^3 + 4q^4 + O(q^5).
$$

\subsubsection{Ratios of successive critical values}
\begin{theorem}For $f_1,f_2 \in S_{13}(\Gamma_0(3),\chi)$ and $g\in S_6(\Gamma_0(3))$ as above, and for an integer $m$ with 
$9 \leq m \leq 11$, one has:
	\begin{align*}
		\frac{L(m, f_1\times g)}{L(m+1,f_1\times g)} \equiv \frac{L(m,f_2\times g)}{L(m+1,f_2\times g)}  \pmod{\mathfrak{P}_{13}}.
	\end{align*}
\end{theorem}
The ratios of the Rankin-Selberg $L$-values are calculated exactly as in the previous subsections. 
The reader should keep in mind that  $f_2^\rho = f_3, f_3^\rho = f_2$, 
$\langle f_3 , f_3 \rangle = \langle f_3^\rho, f_3^\rho \rangle = \langle f_2, f_2 \rangle$, and 
$L_N(s, \chi \psi) = L(s, \psi)$ is the Dirichlet $L$-function for the unique quadratic character modulo $3$. 
The ideal factorization of the difference of the ratios $\frac{L(m, f_1\times g)}{L(m+1,f_1\times g)} - \frac{L(m,f_2\times g)}{L(m+1,f_2\times g)}$, for $m = 11$, 
in the number field $\mathbb{Q}(\nu_1)$ is given by:
\begin{multline*}
 \frac{L(11, f_1\times g)}{L(12,f_1\times g)} - \frac{L(11,f_2\times g)}{L(12,f_2\times g)} = 
 (\left(2, \frac{1}{18} \nu_{1}\right))^{2} \cdot (\left(3, \frac{1}{18} \nu_{1} + 1\right))^{2} \cdot (\left(3, \frac{1}{18} \nu_{1} + 2\right)) \\ 
 \cdot (\left(5, \frac{1}{18} \nu_{1} + 2\right))^{-1} \cdot  (\left(5, \frac{1}{18} \nu_{1} + 3\right))^{-1} \cdot \left(11\right)^{-1} 
 \cdot (\left(13, \frac{1}{18} \nu_{1}\right)) \cdot (\left(71, \frac{1}{18} \nu_{1} + 20\right))^{-1} .
\end{multline*}
So they are congruent modulo $\mathfrak{P}_{13}$; the same conclusion holds for the other ratios as well.

\medskip
\subsection{$S_{26}(\SL_2(\mathbb{Z})) \times S_{13}(\Gamma_0(3),\psi)$}
\label{section: weight-26-x-13}
Unlike the previous examples, here the cusp form of higher weight $k$ is fixed, and the modular forms of lower weight $l$ vary in a congruence class. 

\medskip

For the weight $k = 26$ cusp form, take $f \in S_{26}(\SL_2(\mathbb{Z}))$ with $q$-expansion 
\begin{equation*}
	f = q - 48 q^2 - 195804 q^3 - 33552128 q^4 - 741989850 q^5 + 9398592  q^6 + O(q^7).
\end{equation*}
and view it as an element of $S_{26}(\Gamma_1(3)).$

\medskip

 We take two cusp forms of weight $13$ and level $3$. 
 In order to have consistency of notations call the forms $f_1$ and $f_2$ from \ref{section: weight-13-x-6} as $g_1$ and $g_2$ in $S_{13}(\Gamma_0(3),\psi)$, respectively, 
where $\psi$ is the non-trivial primitive character modulo $3$.  
 It is already established that  $g_1 \equiv g_2 \pmod{ \mathfrak{P}_{13}},$ where $\mathfrak{P}_{13}$ is the prime ideal lying above 
 $(13) \subset \mathbb{Z}$ in $\mathbb{Q}(\nu_1),$ where $\nu_1^2 = -8424$.

\medskip

 As a Hecke module, the $7$-dimensional space $S_{26}(\Gamma_1(3))$ decomposes into a sum of: 
\begin{enumerate}
	\item a $2$-dimensional subspace with a basis of the old forms $f$ and $\hat{f}:=\langle 3 \rangle \cdot f = f(3z),$
	\item a $2$-dimensional subspace with a basis consisting of a newform form and its nontrivial Galois conjugate,
	\item a $3$-dimensional subspace with a basis consisting of a newform form and its two distinct Galois conjugates.
\end{enumerate} 
This data is obtained from the $L$-functions and modular forms database LMFDB \cite{lmfdb}. 
As there are in total $5$ newforms, finding an extension of $\mathbb{Q}(\nu_1)$ containing all the Fourier coefficients of such newforms is computationally taxing. 
Instead, as a work around, we compute a basis consisting of Fourier coefficients in $\mathbb{Q},$ for the space in $(2)$ and for the space in $(3).$ 
This can be achieved in \href{www.sagemath.org}{SAGE} by applying the \texttt{decomposition()} command on any Hecke operator away from the level.

If $h_0$ is the holomorphic projection and $h_0 = \alpha_0 f + \alpha_1 \hat{f} + \dots \in S_{26}(\Gamma_1(3))$ then 
it is \textit{no longer true} that 
$\langle f, h_0 \rangle = c_0 \langle f, f \rangle,$ since $\langle f, \hat{f} \rangle \neq 0.$ Hence $c_r \alpha_0$ is {\it not} the correct $L$-value. 
So while returning the $L$-value this has to be taken into account. The relation between $\langle f, f \rangle $ and $\langle f, \hat{f} \rangle$ is as follows: 
Suppose $f \in S_{k}(\Gamma_1(N))$ be a newform. Let $a(p,f)$ be its coeffecient at $p$. In $S_k(\Gamma_0(Np)),$ taking $\hat{f}(z) = f(pz)$ one has: 
\begin{equation*}
	\langle \hat{f} ,f \rangle_{\Gamma_0(Np)} 
	= \langle f(pz), f(z) \rangle_{\Gamma_0(Np)} 
	= \langle f(z),f(pz) \rangle_{\Gamma_0(Np)} 
	= \frac{p^{-k+1}a(p,f)}{1+p} \langle f,f\rangle_{\Gamma_0(Np)}. 
\end{equation*}
See Bella\"{i}che \cite[p.\,284]{joel}.

\subsubsection{Ratios of special values}
In this example, we verify the congruence in all cases except one recorded in the following theorem:

\begin{theorem}
\label{thm:exceptional-example}
	For $f \in S_{26}(\SL_2(\mathbb{Z}))$ and $g_1, g_2 \in S_{13}(\Gamma_0(3),\psi)$ as above, and an integer $m$ with $19 \leq m \leq 23$, one has:
	\begin{align*}
		\frac{L(m,f\times g_1)}{L(m+1,f\times g_1)}\equiv \frac{L(m,f\times g_2)}{L(m+1,f\times g_2)} \pmod{\mathfrak{P}_{13}}.
	\end{align*}
For the ratios of the right extreme critical values, one has for the Rankin--Selberg convolutions, the congruence:
\begin{align*}
		\frac{D(24,f\times g_1)}{D(25,f\times g_1)} \,\,\,{\equiv}\,\,\, \frac{D(24,f\times g_2)}{D(25,f\times g_2)} \pmod{\mathfrak{P}_{13}}, 
	\end{align*}
however, for the ratios of completed $L$-values, one has:
	\begin{align*}
		\frac{L(24,f\times g_1)}{L(25,f\times g_1)} \,\,\,{\not\equiv}\,\,\, \frac{L(24,f\times g_2)}{L(25,f\times g_2)} \pmod{\mathfrak{P}_{13}}. 
	\end{align*}
	\end{theorem} 

For $19 \leq m \leq 23$, the congruence is verified as in the previous sections. 
In the exceptional case $m = 24$, the non-congruence for the ratios of extreme critical values is explained like this. The character of $f\in S_k(\Gamma_1(3))$ 
is trivial modulo $3$ and the character of $g_1, g_2 \in S_k(\Gamma_0(3), \psi)$ is the non-trivial quadratic character modulo $3$. So
\begin{equation*}
	\frac{L_\infty(24, f\times g_1) L_N(2\cdot 24 + 2 -26-13,\chi\psi)}{L_\infty(25, f\times g_1) L_N(2\cdot 25 + 2 -26-13,\chi\psi)}  
	\ = \ \frac{L_\infty(24, f\times g_2) L(11,\psi)}{L_\infty(25, f\times g_2) L(13,\psi)} = \frac{60951}{444808}.
\end{equation*}

The first few prime ideals (written in the increasing order of their norms) appearing in the ideal factorization of the quantity 
$ \frac{L(24,f\times g_1)}{L(25,f\times g_1)} -\frac{L(24,f\times g_2)}{L(25,f\times g_2)}$ in the number field $\mathbb{Q}(\nu_1)$ is
\begin{multline*}
(\left(2, \frac{1}{18} \nu_{1}\right))^{-9} \cdot (\left(3, \frac{1}{18} \nu_{1} + 1\right))^{2} \cdot (\left(5, \frac{1}{18} \nu_{1} + 2\right))^{-1} \cdot (\left(5, \frac{1}{18} \nu_{1} + 3\right))^{-1} \cdot (\left(7, \frac{1}{18} \nu_{1} + 3\right))^{-2} \\ \cdot (\left(7, \frac{1}{18} \nu_{1} + 4\right))^{-2} \cdot (\left(13, \frac{1}{18} \nu_{1}\right))^{-2} \cdot \left(19\right)^{-1} \cdot (\left(31, \frac{1}{18} \nu_{1} + 6\right))^{-1} \cdots .
\end{multline*}
Hence
\begin{equation}\label{eqn:L24/L25}
\frac{L(24,f\times g_1)}{L(25,f\times g_1)} \ {\not\equiv} \ \frac{L(24,f\times g_2)}{L(25,f\times g_2)} \pmod{\mathfrak{P}_{13}}.
\end{equation}
In fact, 
\begin{equation}\label{eqn:d24/d25}
v_{\mathfrak{P}_{13}}\left(\frac{D(24,f\times g_1)}{D(25,f\times g_1)} - \frac{D(24,f\times g_2)}{D(25,f\times g_2)} \right) = 1.
\end{equation}
The reason for non-congruence is due to the ratio $\frac{L(11,\psi)}{L(13,\psi)}$ of the abelian $L$-functions. Since 
$L(13,\psi)$ has in it the generalized Bernoulli number $B_{13,\psi} = -1445626/3 = -1 \cdot 2 \cdot 3^{-1} \cdot 7 \cdot 13^{3} \cdot 47,$
the congruence modulo $\mathfrak{P}_{13}$ in \eqref{eqn:d24/d25} gets `cancelled' by the $13^3$ in the denominator of $\frac{L(11,\psi)}{L(13,\psi)}$ 
explaining the non-congruence in \eqref{eqn:L24/L25}.

\medskip
\subsection{$S_{24}(\SL_2(\mathbb{Z})) \times M_{12}(\SL_2(\mathbb{Z}))$}
\label{section: weight-24-x-12-ramanujan}
We may also vary the modular forms in a congruence class with one being cuspidal and the other an Eisenstein series. In this section, 
we vary the lower weight modular forms of weight $l = 12$, with one form being the Ramanujan $\Delta$-function, and the other form the weight $12$ Eisenstein series. 
More precisely, let $f$ be the newform of weight $k = 24$ for $\SL_2(\mathbb{Z})$ with $q$-expansion 
$f(q) = q + \beta_{0} q^{2} + \left(-48 \beta_{0} + 195660\right) q^{3} + O(q^{4}),$ where 
$\beta_0$ satisfies the polynomial $x^{2} - 1080 x - 20468736 \in \mathbb{Q}[x]$. This is one of the newforms from Section \ref{section: weight-24-x-12}. 
For the smaller weight $l = 12$, take $g_1: = \Delta \in S_{12}(\SL_2(\mathbb{Z}))$ and $g_2: = E_{12}(z) ={691/65520}+q+\cdots \in M_{12}(\SL_2(\mathbb{Z}))$, 
the Eisenstein series of weight $12$ and level $1$. The well-known Ramanujan congruence asserts that $ g_1 \equiv g_2 \pmod{691}.$
 
\begin{theorem} \label{thm:ramanujan-cong}
For $f \in S_{24}(\SL_2(\mathbb{Z}))$ and $g_1, g_2 \in M_{12}(\SL_2(\mathbb{Z}))$ as above
	\begin{align*}
		\frac{L(m,f\times g_1)}{L(m+1,f\times g_1)} \ \equiv \ \frac{L(m,f\times g_2)}{L(m+1,f\times g_2)}  \pmod{\mathfrak{P}_{691}}\:\:\:\mbox{for}\:\:m=22,21,20,19,18.
	\end{align*}
\end{theorem}
The $p$-th Fourier coefficient of $E_{12}$ being $p^{11}+1$, one has $L(s, f \times E_{12}) = L(s,f)L(s+11,f)$, 
it is probably more enlightening to write the above congruence as: 
$$
\frac{L(m,f\times \Delta)}{L(m+1,f\times \Delta)} \ \equiv \ \frac{L(m,f) L(m-11,f)}{L(m+1,f)L(m-10,f)}  \pmod{\mathfrak{P}_{691}}
$$
In a similar vein, by Shimura \cite[Lem.\,1]{shimura}, one has:
\begin{equation*}
	D(s,f\times g_2) = \frac{D(s,f)\cdot D(s-11,f)}{\zeta(2s-34)}, \label{eqn: product-of-smaller-L-functions}
\end{equation*}
where $\zeta(s)$ is the Riemann zeta function and $D(s,f) = \sum_{n=1}^{\infty} a(n,f) n^{-s}$. Therefore,
\begin{align*}
	\dfrac{D(m,f,g_2)}{D(m+1,f,g_2)} 
	&= \frac{\zeta(2m+2 -34)}{\zeta(2m -34)}\dfrac{D(m , f)}{D(m -10 , f)} \dfrac{D(m-11, f)}{ D(m + 1, f)}\\
	&= \left(\frac{1}{(2 \pi \sqrt{-1})^2} \frac{\zeta(2m+2 -34)}{\zeta(2m -34)}\right) \cdot \left(\dfrac{D(m , f)}{(2 \pi \sqrt{-1})^{-1}D(m -10 , f)} \right) \\
	&\qquad\qquad \qquad \cdot \left( \dfrac{D(m-11, f)}{(2 \pi \sqrt{-1})^{-1} D(m + 1, f)} \right).
\end{align*}
The first term inside $\left( \cdots \right)$ is rational due to well known result on the special values of the Riemann zeta function. 
The second and the third term are algebraic due to equations \eqref{eqn: odd-l-value} and \eqref{eqn: even-l-value}. Also, the reader is reminded that the form $f$ has totally real coefficients.
Hence we use the algorithms described in \ref{subsubsection: algorithm-gl2-even} and \ref{subsubsection: algorithm-gl2-odd} to calculate the ratios of the special values 
$\frac{L(m, f\times g_2)}{L(m+1, f\times g_2)}.$
The ratios $\frac{D(m,f)}{(-2\pi \sqrt{-1})^{m-1}D(1,f)}$ and $\frac{D(m,f)}{(-2\pi \sqrt{-1})^{m-2} D(2,f)}$ depending on the pairity of $m$ are summarized in the table.  

\medskip
\def\arraystretch{1.5}
\begin{tabular}{|c|c|}
	\hline
	$m$ & $\frac{D(m,f)}{(-2\pi \sqrt{-1})^{m-1} D(1,f)}$\\
	\hline
	\hline
	$1$ & $1$\\
	\hline
	
	$3$ & $-\frac{569}{18825760800} \beta_{0} - \frac{16757416}{196101675}$\\
	\hline
	$5$ & $\frac{3403}{218993544000} \beta_{0} + \frac{40511069}{4562365500}$\\
	\hline
	$7$ & $-\frac{147089}{17869873190400} \beta_{0} - \frac{854671739}{744578049600}$\\
	\hline
	$9$ & $\frac{6538127}{1250891123328000} \beta_{0} + \frac{10060850717}{52120463472000}$\\
	\hline
	$11$ & $-\frac{320477}{77436117158400} \beta_{0} - \frac{169463087}{3226504881600}$ \\
	\hline
	$13$ & $\frac{320477}{77436117158400} \beta_{0} + \frac{169463087}{3226504881600}$\\
	\hline
	$15$ & $-\frac{6538127}{1250891123328000} \beta_{0} - \frac{10060850717}{52120463472000}$\\
	\hline
	$17$ & $\frac{147089}{17869873190400} \beta_{0} + \frac{854671739}{744578049600}$\\
	\hline
	$19$ &  $-\frac{3403}{218993544000} \beta_{0} - \frac{40511069}{4562365500}$\\
	\hline
	$21$ &  $\frac{569}{18825760800} \beta_{0} + \frac{16757416}{196101675}$\\
	\hline
	$23$ &  $-1$\\
	\hline
\end{tabular}
\quad
\begin{tabular}{|c|c|}
 		\hline
 		$m$ & $\frac{D(m,f)}{(-2\pi \sqrt{-1})^{m-2} D(2,f)}$\\
 		\hline
 		\hline
 		$2$ & $1$\\
 		\hline
 		$4$ & $-\frac{1}{15120000} \beta_{0} - \frac{59221}{630000}$\\
 		\hline
 		$6$ & $\frac{1}{26593920} \beta_{0} + \frac{12031}{1108080}$\\
 		\hline
 		$8$ & $-\frac{1783}{80372736000} \beta_{0} - \frac{5312293}{3348864000}$\\
 		\hline
 		$10$ & $\frac{37}{2344204800} \beta_{0} + \frac{31327}{97675200}$\\
 		\hline
 		$12$ & $-\frac{847}{60279552000} \beta_{0} - \frac{371437}{2511648000}$ \\
 		\hline
 		$14$ & $\frac{37}{2344204800} \beta_{0} + \frac{31327}{97675200}$\\
 		\hline
 		$16$ & $-\frac{1783}{80372736000} \beta_{0} - \frac{5312293}{3348864000}$\\
 		\hline
 		$18$ & $\frac{1}{26593920} \beta_{0} + \frac{12031}{1108080}$\\
 		\hline
 		$20$ &  $-\frac{1}{15120000} \beta_{0} - \frac{59221}{630000}$\\
 		\hline
 		$22$ &  $1$\\
 		\hline
\end{tabular} 

\medskip

From the table above one concludes
\begin{align*}
	\frac{L(22,f\times g_2)}{L(23,f\times g_2)} &= -\frac{19153}{3052249200} \beta_{0} + \frac{23359724}{63588525}.
\end{align*}
Similarly, for the other ratios in the right hand side of Thm.\,\ref{thm:ramanujan-cong} are determined.
For the left hand side of Thm.\,\ref{thm:ramanujan-cong}, note that  $\frac{D(m,f \times g_1)}{\pi^{24} \langle f \times f\rangle}$ and 
the ratios of successive critical values $\frac{L(m, f\times g_1)}{L(m+1, f\times g_1)}$ have already been calculated in Sect.\,\ref{section: weight-24-x-12}.
The prime ideal factorization of the difference of the ratios has the following shape: 
\begin{multline*}
\frac{L(22,f\times g_1)}{L(23,f\times g_1)} - \frac{L(22,f\times g_2)}{L(23,f\times g_2)} =
\cdots (\left(13, \frac{1}{24} \beta_{0} - 20\right))^{-1} \cdot (\left(73, \frac{1}{24} \beta_{0} - 6\right)) \\ \cdot (\left(691, \frac{1}{24} \beta_{0} - 12\right)) \cdot (\left(691, \frac{1}{24} \beta_{0} + 658\right)) \cdot (\left(23003, \frac{1}{24} \beta_{0} + 2325\right)).
\end{multline*}
The only prime that ramifies in $\mathbb{Q}(\beta_0)$ is $144169.$ The ideal $\mathfrak{P}_{691} := (691) \subset \mathbb{Q}(\beta_0)$ is split 
and has the factorization $ (\left(691, \frac{1}{24} \beta_{0} - 12\right)) \cdot (\left(691, \frac{1}{24} \beta_{0} + 658\right))$. Hence
\begin{equation*}
	\frac{L(22,f\times g_1)}{L(23,f\times g_1)} \equiv \frac{L(22,f\times g_2)}{L(23,f\times g_2)} \pmod{\mathfrak{P}_{691}}.
\end{equation*}
Same conclusion holds for other ratios as well.

\bigskip
\section{A conjecture on the ratios of the special values of the Rankin-Selberg L-functions}
The main purpose of this article is to investigate the principle:
$$
f \equiv f' \pmod{\fP} \ \ \implies \ \ \frac{L(m, f \times g)}{L(m+1, f \times g)} \equiv \frac{L(m, f' \times g)}{L(m+1, f' \times g)} \pmod{\fP}.
$$
Keeping the above examples in mind which lend credence to this principle, while 
accounting for the exceptional situation as in Thm.\,\ref{thm:exceptional-example}, we formulate the following conjecture:

\begin{conjecture}
\label{conj:main}
	Let $f, f' \in S_k(N,\chi)$  and $g \in S_l(M, \psi)$ be normalized eigenforms with $k - l >2$. 
	Assume that for a prime ideal $\mathfrak{P}$ in the compositum of the number fields $\mathbb{Q}(f),$ $\mathbb{Q}(f'),$ and $\mathbb{Q}(g)$ 
	one has the congruence: 
	\begin{equation*} a(n,f) \equiv a(n,f') \pmod{\mathfrak{P}}\;\;\text{ for all }n\in\mathbb{N}.
	\end{equation*}
For an integer $m$ with $
\lceil \frac{k+l-1}{2} \rceil \leq m,m+1 \leq k-1,$ assume: 
	\begin{equation}
	\label{hyp: hypotheis-on-the-prime-ideal}
	v_{\mathfrak{P}}\left( 
	\frac{L_\infty(m, f\times g) L_N(2m+2-k-l, \chi \psi)}{L_\infty(m+1, f\times g) L_N(2m+4-k-l, \chi \psi)}
	\right) \geq 0.
	\end{equation}
	Then for all such $m$ one has the congruence:
	\begin{align}
		\frac{L(m,f\times g)}{L(m+1,f\times g)} \ \equiv \ \frac{L(m,f'\times g)}{L(m+1,f'\times g)} \pmod{\mathfrak{P}}. \label{eqn: main-question}
	\end{align}
	 Now, let $f \in S_k(N,\chi)$ and $g, g' \in S_l(M,\psi)$ be normalized eigenforms and $\mathfrak{P} \subset \mathbb{Q}(f)\mathbb{Q}(g')\mathbb{Q}(g)$ be a prime ideal such that
	\begin{equation*}
		a(n,g) \equiv a(n,g') \pmod{\mathfrak{P}} \text{ for  all } n \in \mathbb{N}.
	\end{equation*}
	Assume \eqref{hyp: hypotheis-on-the-prime-ideal} is satisfied. Then 
	\begin{align}
		\frac{L(m,f\times g)}{L(m+1,f\times g)} \ \equiv \ \frac{L(m,f\times g')}{L(m+1,f\times g')} \pmod{\mathfrak{P}}.\label{eqn: dual-main-question}
	\end{align} \label{conjecture: main-conjecture}
\end{conjecture}

In the preceding section, we computationally verified this Conjecture \ref{conjecture: main-conjecture} in many examples. 
Especially, note that the hypothesis \eqref{hyp: hypotheis-on-the-prime-ideal} has been introduced to take care of the exceptional situation in 
Thm.\,\ref{thm:exceptional-example}. As mentioned in the introduction, in the companion article \cite{narayanan-raghuram}, we prove a 
variation of this conjecture using the machinery of Eisenstein cohomology as developed by 
Harder and the second author in \cite{harder-raghuram}.

\bigskip

\noindent {\it Acknowledgements:} 
The authors are grateful to the Institute for Advanced Study, Princeton, for a summer collaborator's grant in 2023 when this project got started. 
The authors also thank Baskar Balasubramanyam for some invaluable comments. The calculations were carried out at the Bhaskara Lab of IISER Pune. 
The first author is supported by the CSIR fellowship for his Ph.D.

\bigskip

\bigskip

\end{document}